\documentclass{amsart}
\usepackage[utf8]{inputenc}

\usepackage{amssymb,tikz}
\usepackage{latexsym}
\usepackage{amsmath,mathdots}
\usepackage{euscript}
\usepackage{graphics}
\usepackage{todonotes}
\usepackage{pgf,tikz,pgfplots}
\pgfplotsset{compat=1.14}
\usepackage{mathrsfs}
\usetikzlibrary{arrows}
      \def\dR{{\mathbb R}}

\def\bm\chi{\mbox{\boldmath$\chi$}}

\let\xker=\ker \def\ker{{\xker\,}}

\newtheorem{theorem}{Theorem}[section]
\newtheorem{proposition}[theorem]{Proposition}
\newtheorem{corollary}[theorem]{Corollary}

\theoremstyle{remark}

\newtheorem{remark}[theorem]{Remark}

\numberwithin{equation}{section}
\begin{document}

\date{\today}

\author[M.~Derevyagin]{Maxim~Derevyagin}
\address{
MD,
Department of Mathematics\\
University of Connecticut\\
341 Mansfield Road, U-1009\\
Storrs, CT 06269-1009, USA}
\email{maksym.derevyagin@uconn.edu}

\author[N.~Juricic]{Nicholas~Juricic}
\address{
NJ,
Department of Mathematics\\
University of Connecticut\\
341 Mansfield Road, U-1009\\
Storrs, CT 06269-1009, USA}
\email{nicholas.juricic@uconn.edu}


%%%%%%%%%%%%%%%%%%%%%%%%%%%%%%%%%%%%%%%%%%%%%%%%%%%%%%%%%%%%%%%%%%%%%%%%%%%%%%%%%%

\title[Asymptotics of integrals of products of Jacobi polynomials]{A note on an asymptotic formula for integrals of products of Jacobi polynomials}

\begin{abstract}
We recast Byerly's formula for integrals of products of Legendre polynomials. Then we adopt the idea to the case of Jacobi polynomials. After that, we use the formula to derive an asymptotic formula for integrals of products of Jacobi polynomials. The asymptotic formula is similar to an analogous one recently obtained by the first author and Jeff Geronimo for a different case. Thus, it suggests that such an asymptotic behavior is rather generic for integrals of products of orthogonal polynomials. 
\end{abstract}

\maketitle

\section{Introduction}

In \cite{GM15} the Alpert multiresolution analysis was studied and important in this study was the integral
\[
f_{n,m}=\int_0^1 \hat p_n(t)\hat p_m(2t-1)dt,
\]
where $\hat p_n$ is the orthonormal Legendre polynomial to be precisely defined in the next section. These coefficients are entries in the refinement equation associated with the multiresolution analysis and they have been shown to satisfy a variety of difference equations. 
As a matter of fact, it turns out that more general coefficients of the form
\[
u_{n,m}=\int_{\dR} P_n(t)Q_m(\alpha t+\beta) d\sigma(t),
\]
where $\{P_n\}_{n=0}^{\infty}$ and $\{Q_n\}_{n=0}^{\infty}$ are two families of orthonormal polynomials, $\sigma$ is a measure on $\dR$ with finite moments, and the numbers $\alpha\ne 0$, $\beta$ are complex, satisfy a generalized wave equation on the two dimensional lattice
\begin{equation*}
a_{n+1}u_{n+1,m}+b_nu_{n,m}+a_nu_{n-1,m}
=\frac{c_{m+1}}{\alpha}u_{n,m+1}+\frac{d_{m}-\beta}{\alpha}u_{n,m}+\frac{c_m}{\alpha}u_{n,m-1}
\end{equation*}
for $n,m=0$, $1$, $2$, \dots (see \cite{DG} for details).
It was also observed in \cite{DG} that a damped oscillatory behavior took place for such coefficients. In particular, for the coefficients
\[
f^{(\lambda)}_{n,m}=\int_0^1 \hat p^{(\lambda)}_n(t)\hat p^{(\lambda)}_m(2t-1)(t(1-t))^{\lambda-1/2}dt,
\] 
where ${\hat p^{(\lambda)}_n}$ are the orthonormal ultraspherical polynomials and $\lambda>\ -1/2$, it was shown that  
\begin{equation}\label{LegAs}
f^{(\lambda)}_{n,m}=k_m\frac{\cos\left(\pi\left(m + \frac{\lambda}{2} - \frac{n}{2} + \frac{1}{4}\right)\right)}{\sqrt{\pi}n^{\lambda + 1/2}}+O\left(\frac{1}{n^{\lambda+3/2}}\right),
\end{equation}
as $n\to\infty$, where the coefficient $k_m$ can be explicitly found in terms of $m$.

The goal of the present paper is to obtain an asymptotic formula similar to \eqref{LegAs} for the integrals
\[
 \int_{x}^1 P_n^{(\alpha,\beta)}(t)P_m^{(\alpha,\beta)}(t)\, (1-t)^{\alpha}(1+t)^{\beta}dt,
\]
as $n\to\infty$, where $x$, $m$ are fixed, and $P_n^{(\alpha,\beta)}(x)$  are Jacobi polynomials. First, we recast a useful formula by Byerly in the case of Legendre polynomials and then we will derive the desired formula.

\section{Legendre polynomials}

Recall that the Legendre polynomial $P_n$ of degree $n$ is a polynomial solution to the second order differential equation
\begin{equation}\label{LegendreDE}
(1-x^2)y''(x)-2xy'(x)+n(n+1)y(x)=0.
\end{equation}
Actually, it is not so hard to check that for every nonnegative integer $n$ equation \eqref{LegendreDE} has a unique polynomial solution up to a multiplicative constant.  To be definite, throughout this section we assume that $P_n$ is a monic polynomial.

%\textcolor{red}{
%\[
%P_n(x)=\hypergeom21{-n,n+1}{1}{\frac{1-x}{2}}
%\]}
%is a polynomial solution to \eqref{LegendreDE}. Recall that the hypergeometric function
%\begin{equation} \label{HypergeometricFunction} 
%    \hypergeom21{a,b}{c}{z} = \sum_{n=0}^{\infty}\frac{(a)_n(b)_n}{(c)_n}\frac{z^n}{n!}
%\end{equation}
%is a solution of the hypergeometric differential equation 
%\begin{equation} \label{HypergeometricDE}
%    z(1-z)\frac{d^2y}{dz^2} + [c - (a+b+1)z]\frac{dy}{dz} = 0.
%\end{equation}
%By making the change of variables $x = 1-2z$ in (\ref{LegendreDE}), one finds that the resulting equation is written in the form of (\ref{HypergeometricDE}) with $a = -n$, $b = n + 1$, and $c = -1$.

In his book published in 1893, W.E. Byerly noted the following relation for Legendre polynomials that showed itself to be useful in a few instances where Legendre polynomials appear. 

\begin{proposition}[See p. 172 in \cite{B59}] \label{thm:orthogonal1}
Let $n$ and $m$ be nonnegative integers such that $n\ne m$. Then we have that
\begin{equation}\label{IPL1}
 \int_{x}^1P_n(t)P_m(t)\,dt=\frac{(1-x^2)\left[P_m(x){P'_n(x)}-P_n(x){P'_m(x)}\right]}{n(n+1)-m(m+1)},   
\end{equation}
where $x$ is any real number.
\end{proposition}
\begin{proof}
Evidently, the polynomials $P_n$ and $P_m$ satisfy the differential equations
\begin{equation}\label{LegDEHelp1}
    \frac{d}{dt}\left[(1-t^2)\frac{dP_n(t)}{dt} \right] = -n(n+1)P_n(t) 
\end{equation}
and 
\begin{equation}\label{LegDEHelp2}
    \frac{d}{dt}\left[(1-t^2)\frac{dP_m(t)}{dt} \right]=- m(m+1)P_m(t).
\end{equation}
Multiplying \eqref{LegDEHelp1} by $P_m$ and \eqref{LegDEHelp2} by $P_n$, respectively, and then subtracting them and integrating the result gives

\begin{gather} \label{eqn:legendreintegral1}
     \left[ n(n+1) - m(m+1)\right] \int_{x}^1 P_n(t)P_m(t)\, dt = \\ 
     =\int_{x}^1 P_n(t)\frac{d}{dt}\left[(1-t^2)\frac{dP_m(t)}{dt}\right] \, dt - \int_{x}^1 P_m(t)\frac{d}{dt}\left[(1-t^2)\frac{dP_n(t)}{dt}\right] \, dt.\nonumber
     \end{gather}
After integrating each of the integrals on the right-hand side of (\ref{eqn:legendreintegral1}) by parts we arrive at \eqref{IPL1}.
    \end{proof}

 Formula \eqref{IPL1} immediately implies the orthogonality of Legendre polynomials.
    
\begin{corollary}
The polynomials $P_n$ and $P_m$ are orthogonal with respect to the Lebesgue measure on $[-1,1]$ provided that $n\ne m$, that is, 
\[
\int_{-1}^1 P_n(t)P_m(t)\,dt=0.
\]
\end{corollary}
\begin{proof}
It directly follows from \eqref{IPL1} if we set $x=-1$.
\end{proof}

Another aspect of formula \eqref{IPL1} is that it shows a relation between integrals of products of orthogonal polynomials and the \lq\lq augmented Wronskian" introduced by Karlin and Szeg\H{o}, who were motivated by some probabilistic problems \cite{KS60}. Namely, they studied the polynomials
\[
\varphi_{n}(m; x)=\begin{vmatrix}
P_m(x)&P_n(x)\\
P'_m(x)&P'_n(x)
\end{vmatrix}.
\]
In these notations, formula \eqref{IPL1} reads that
\[
\varphi_{n}(m; x)=\frac{n(n+1)-m(m+1)}{1-x^2}\int_{x}^1P_n(t)P_m(t)\,dt
\]
which gives an integral representation of polynomials $\varphi_n(m; x)$. In addition, Karlin and Szeg\H{o} showed that the polynomials {\it $\varphi_n(m; x)$ satisfy a second order differential equation} \cite[Chapter 4]{KS60}. At the same time, in the literature on integrable systems the family of polynomials $\varphi_n(m; x)$ is called the Darboux transformation of the system $P_n(x)$ (for instance, see \cite{MS91}) and is known to produce solvable equations if the original equation is solvable.   

%\begin{corollary} \label{IPL2}
%If $n,m \geq 1$ and $n \neq m$, then we have
%\begin{equation}
%\int_{-1}^1 \left( \int_x^1 P_n(t)\,dt \int_x^1 P_m(t)\, dt \right) \frac{dx}{1-x^2} = 0.
%\end{equation}
%\end{corollary}
%
%\begin{proof}
%First note by polynomial division that the integral in (\ref{IPL2}) converges because $m,n \geq 1$. Noting that $P_0(x) \equiv 1$, we apply Theorem~\ref{IPL1} to deduce the identity
%\begin{equation}
%\int_x^1 P_n(t)\, dt = \frac{(1-x^2)P_n'(x)}{n(n+1)}, 
%\end{equation}
%and so the integral in (\ref{IPL2}) becomes 
%\begin{equation*} \label{equation1}
%    \frac{1}{n(n+1)}\int_{-1}^1 \left( P_n'(x) \int_x^1 P_m(t)\, dt \right) \, dx.
%\end{equation*}
%Now we integrate by parts and use the fundamental theorem of calculus on $\int_x^1 P_m(t)\, dt$, and (\ref{IPL2}) is given by
%\begin{equation*}
%    \frac{1}{n(n+1)}\left(P_n(x)\int_x^1 P_m(t)\, dt\right)\Big|_{x = -1}^{x=1} + \frac{1}{n(n+1)}\int_{-1}^{1}P_n(x)P_m(x)\, dx = 0.
%\end{equation*}
%Note that $\int_{-1}^1 P_m(t)\,dt = 0$ since $m > 0$, and $\int_{-1}^{1}P_n(x)P_m(x)\, dx = 0$ as $n \neq m$.
%\end{proof}

%Another application of formula \eqref{IPL1} is the following result.
%
%\begin{corollary}
%If $n$ and $m$ are two nonnegative integers such that $n\ne m$ then \textcolor{red}{
%\[
%\int_{0}^1 P_n(t)P_m(t)\,dt=???
%\]} 
%\end{corollary}

\section{Jacobi Polynomials}

In this section we will apply Byerly's idea to the Jacobi differential equation.  To this end, let $P_n^{(\alpha,\beta)}(x)$ be a polynomial solution of the Jacobi differential equation
\begin{equation} \label{eqn:jacobidiffeq}
    (1-x^2)\frac{d^2y}{dx^2} + (\beta - \alpha - (\alpha + \beta + 2)x)\frac{dy}{dx} + n(n + \alpha + \beta + 1)y = 0,
\end{equation}
where we assume that $\alpha$, $\beta>-1$. It is clear that such a definition determines a polynomial up to a constant factor. So, to be more precise, one can check that 
\begin{equation}\label{JacobiExplicit}
P_n^{(\alpha,\beta)}(x)=\frac{(\alpha+1)_n}{n!}\sum_{k=0}^{n}\frac{(-n)_k(\alpha+\beta+n+1)_k}{(\alpha+1)_k}\left(\frac{1-x}{2}\right)^k
\end{equation}
is a solution to \eqref{eqn:jacobidiffeq}, where $(a)_k$ is the Pochhammer symbol, that is, $(a)_0=1$ and $(a)_k=a(a+1)...(a+k-1)$ for $k=1, 2, 3, \dots$. Besides, formula \eqref{JacobiExplicit} defines the Jacobi polynomial $P_n^{(\alpha,\beta)}(x)$ uniquely. 

Next, following Karlin and Szeg\H{o} \cite{KS60} let us introduce the \lq\lq augmented Wronskian" of the form
\[
\varphi_{n}(m; \alpha, \beta, x)=\begin{vmatrix}
P_m^{(\alpha,\beta)}(x)&P_n^{(\alpha,\beta)}(x)\\
\dfrac{dP_m^{(\alpha,\beta)}(x)}{dx}&\dfrac{dP_n^{(\alpha,\beta)}(x)}{dx}
\end{vmatrix}.
\]

\begin{remark} The polynomials $\varphi_{n}(m; \alpha, \beta, x)$ are in fact the Darboux transform of $P_n^{(\alpha,\beta)}$ (for the definition of the Darboux transform see \cite{MS91}). Also, the polynomials $\varphi_{n}(m; \alpha, \beta, x)$ are a particular case of the generalized Jacobi polynomials introduced in \cite{B19}. 
\end{remark}
 
\begin{theorem}
Let $\alpha$, $\beta>-1$. Then, for any nonnegative integers $n$ and $m$ such that $n\ne m$, we have
\begin{equation} \label{IPL3}
    \int_x^1 P_n^{(\alpha,\beta)}(t)P_m^{(\alpha,\beta)}(t)\, w(t) dt = 
    \frac{w(x)(1-x^2)\varphi_{n}(m; \alpha, \beta, x)}{n(n+\alpha+\beta+1)-m(m+\alpha+\beta+1)},  
\end{equation}
where $w(t) = (1-t)^{\alpha}(1+t)^{\beta}$ and $x$ is any real number.
\end{theorem}

\begin{proof}
By definition, the polynomials $P_n^{(\alpha,\beta)}$ and $P_m^{(\alpha,\beta)}$ satisfy the differential equations 
\begin{equation} \label{eqn:jacobi1}
    (1-t^2)\frac{d^2P_n^{(\alpha,\beta)}(t)}{dt^2} + (\beta - \alpha - (\alpha + \beta + 2)t)\frac{dP_n^{(\alpha,\beta)}(t)}{dt} + n(n+\alpha + \beta + 1)P_n^{(\alpha,\beta)}(t) = 0
\end{equation}
and 
\begin{equation} \label{eqn:jacobi2}
    (1-t^2)\frac{d^2P_m^{(\alpha,\beta)}(t)}{dt^2} + (\beta - \alpha - (\alpha + \beta + 2)t)\frac{dP_m^{(\alpha,\beta)}(t)}{dt} + m(m+\alpha + \beta + 1)P_m^{(\alpha,\beta)}(t) = 0.
\end{equation}
We multiply (\ref{eqn:jacobi1}) by $P_m^{(\alpha,\beta)}$ and (\ref{eqn:jacobi2}) by $P_n^{(\alpha,\beta)}$, respectively, and then subtract the results to get a differential equation for $P_n^{(\alpha,\beta)}(t)P_m^{(\alpha,\beta)}(t)$. 
%Define a function by 
%\begin{equation}
%    y(t) = P_m^{(\alpha,\beta)}(t)\dfrac{dP_n^{(\alpha,\beta)}(t)}{dt}-P_n^{(\alpha,\beta)}(t)\dfrac{dP_m^{(\alpha,\beta)}(t)}{dt}.
%\end{equation}
After using the product rule on $\varphi_{n}(m; \alpha, \beta, x)$ and rewriting the differential equation becomes 
\begin{equation}\label{equation2}
     -CP_n^{(\alpha,\beta)}(t)P_m^{(\alpha,\beta)}(t) = (1-t^2)\varphi'_{n}(m; \alpha, \beta, t) + (\beta - \alpha - (\alpha + \beta + 2)t)\varphi_{n}(m; \alpha, \beta, t),
\end{equation}
where $C=n(n+\alpha+\beta+1)-m(m+\alpha+\beta+1)$.
We now multiply (\ref{equation2}) by $w(t)$ and integrate over $[x,1]$ to obtain
\begin{equation*}
\begin{split}
    -C\int_x^1P_n^{(\alpha,\beta)}(t)P_m^{(\alpha,\beta)}(t)\,w(t)dt = \int_x^1 (1-t^2)\varphi'_{n}(m; \alpha, \beta, t)\,w(t)dt \\
    + \int_x^1(\beta - \alpha - (\alpha + \beta + 2)t)\varphi_{n}(m; \alpha, \beta, t)\,w(t)dt.
  \end{split}
\end{equation*}
After handling the left integral on the right-hand side by parts, we are left with 
\begin{equation} \label{generalweight}
\begin{split}
    -C\int_x^1P_n^{(\alpha,\beta)}(t)P_m^{(\alpha,\beta)}(t)\,w(t)dt = -(1-x^2)w(x)\varphi_{n}(m; \alpha, \beta, x) \\
    -\int_x^1 \varphi_{n}(m; \alpha, \beta, t)(1-t^2)\, w'(t)dt 
    +     \int_x^1 (\beta - \alpha - (\alpha + \beta)t)\varphi_{n}(m; \alpha, \beta, t)\, w(t)dt.
    \end{split}
\end{equation}
Then one can check that 
\begin{equation*}
  \int_x^1 \varphi_{n}(m; \alpha, \beta, t)(1-t^2)\, w'(t)dt = \int_x^1 (\beta - \alpha - (\alpha + \beta)t)\varphi_{n}(m; \alpha, \beta, t)\, w(t)dt, 
\end{equation*}
which completes the proof.
%\begin{equation*}
%    \int_t^1P_n^{(\alpha,\beta)}(t)P_m^{(\alpha,\beta)}(t)\,w(t)dt = \frac{w(x)(1-x^2)}{C}\left[P_m^{(\alpha,\beta)}(x)\dfrac{dP_n^{(\alpha,\beta)}(x)}{dx}-P_n^{(\alpha,\beta)}(x)\dfrac{dP_m^{(\alpha,\beta)}(x)}{dx}\right].
%\end{equation*}
\end{proof}

As in the case of Legendre polynomials, formula \eqref{IPL3} immediately leads to orthogonality.

\begin{corollary}\label{jacobicorollary}
Let $\alpha$, $\beta>-1$. Then the polynomials $P_n^{(\alpha,\beta)}$ and $P_m^{(\alpha,\beta)}$ are orthogonal with respect to the weight $w(x) = (1-x)^\alpha(1+x)^\beta$ on $[-1,1]$ provided that $n\ne m$, that is, 
\begin{equation*}
    \int_{-1}^1P_n^{(\alpha,\beta)}(t)P_m^{(\alpha,\beta)}(t)\,(1-t)^\alpha(1+t)^\beta\, dx = 0.
\end{equation*}
\end{corollary}

%\begin{remark}
%In our proof of Theorem~\ref{IPL3}, we are careful to treat the weight $w$ without regarding its specific definition until the very end of the argument. Using this approach, it is possible to generalize the statement to an arbitrary continuously differentiable function $w$. The result still involves the term $(1-x^2)w(x)y(x)$, but includes two more integral terms as shown in (\ref{generalweight}). In the particular case $w(x) = (1-x)^\alpha (1+x)^\beta$, it happens that these integral terms are negatives. Similar formulas can be obtained for all classical orthogonal polynomials.
%\end{remark}
The following result is another consequence of formula \eqref{IPL3}, and a well-known form of this result will be used in the next section. 

\begin{corollary}
Let $\alpha$, $\beta>-1$ and let $w(t) = (1-t)^{\alpha}(1+t)^{\beta}$. Then for any $n,m \geq 1$ and $n \neq m$, we have 
\begin{equation} \label{IPL4}
    \int_{-1}^1 \left( \int_x^1 P_n^{(\alpha,\beta)}(t)w(t)dt \int_x^1 P_m^{(\alpha,\beta)}(t)w(t)dt \right) \frac{dx}{(1-x^2)w(t)} = 0.
\end{equation}
\end{corollary}

\begin{remark}
It follows from \eqref{IPL3} that
\[
\int_x^1 P_n^{(\alpha,\beta)}(t)w(t)dt =\frac{(1-x^2)w(x)\varphi_{n}(0; \alpha, \beta, x)}{n(n+\alpha+\beta+1)}
\]
and thus \eqref{IPL4} can be rewritten in the following manner
\begin{equation}\label{phiOrt}
\int_{-1}^1\varphi_{n}(0; \alpha, \beta, t)\varphi_{m}(0; \alpha, \beta, t)(1-t^2)w(t)dt=0,
\end{equation}
which in particular shows that the integral from \eqref{IPL4} exists. 
\end{remark}

\begin{proof}
To apply formula~\eqref{IPL3}, note that $n > 0$ and $P_0^{(\alpha,\beta)} \equiv 1$ to get the identity
\begin{equation}
    \int_x^1 P_n^{(\alpha,\beta)}(t) \, w(t) dt = \frac{(1-x^2)w(x)(P_n^{\alpha,\beta)}(x))'}{n(n+\alpha+\beta+1)}.
\end{equation}
Substituting this expression into (\ref{IPL4}) yields a cancellation, leaving
\begin{equation*}
    \frac{1}{n(n+\alpha+\beta+1)}\int_{-1}^1 \left( (P_n^{(\alpha,\beta)}(x))' \int_x^1 P_m^{(\alpha,\beta)}(t)\,w(t)dt \right)dx.
\end{equation*}
After integrating by parts and using the fundamental theorem of calculus on the term
\[
\int_x^1 P_m^{(\alpha,\beta)}(t)\,w(t)dt, 
\]
the integral in (\ref{IPL4}) becomes (up to a constant factor which is not important since we have 0 on the right-hand side) 
\begin{equation*}
    \left(P_n^{(\alpha,\beta)}(x)\int_x^1 P_m^{(\alpha,\beta)}(t)\,w(t)dt\right) \Big|_{x=-1}^{x=1} +\int_{-1}^{1} P_n^{(\alpha,\beta)}(x)P_m^{(\alpha,\beta)}(x)\, w(x)dx = 0.
\end{equation*}
Observe that $\int_{-1}^{1} P_m^{(\alpha,\beta)}(t)\,w(t)dt = 0$ and $\int_{-1}^{1} P_n^{(\alpha,\beta)}(x)P_m^{(\alpha,\beta)}(x)\, w(x)dx = 0$ using Corollary~\ref{jacobicorollary}.
\end{proof}
\begin{remark}
Since 
\[
\varphi_{n}(0; \alpha, \beta, x)=(P_n^{(\alpha,\beta)}(x))'
\]
formula \eqref{phiOrt} implies the following well-known formula
\begin{equation}\label{shift}
    \frac{d}{dx}\left[ P_n^{(\alpha,\beta)}(x) \right] = \left( \frac{n + \alpha + \beta + 1}{2} \right)P_{n-1}^{(\alpha+1,\beta+1)}(x).
\end{equation}
Here we are taking into account the normalization given by \eqref{JacobiExplicit}.
\end{remark}

\section{Asymptotic formula for the integrals in question}

In this section we will obtain an asymptotic formula for the integral 
\[
\int_x^1 P_n^{(\alpha,\beta)}(t)P_m^{(\alpha,\beta)}(t)\, w(t) dt
\]
when $x$ and $m$ are fixed and $n\to\infty$. To begin with, let us recall a classical asymptotic result for Jacobi polynomials.
\begin{proposition}[\cite{I09}] \label{mourad}
Let $\alpha,\beta > -1$ and set 
\begin{equation*}
    N = n + \frac{\alpha + \beta + 1}{2},\quad \gamma = -\frac{(\alpha + 1/2)\pi}{2}.
\end{equation*}
Then for a fixed $\theta\in(0,\pi)$, we have 
\begin{equation}
    P_n^{(\alpha,\beta)}(\cos\theta) = \frac{k(\theta)}{\sqrt{n}}\cos(N\theta + \gamma) + O(n^{-3/2}),
\end{equation}
as $n\to \infty$, where
\begin{equation}
    k(\theta) = \frac{1}{\sqrt{\pi}}[\sin(\theta/2)]^{-\alpha - 1/2}[\cos(\theta/2)]^{-\beta - 1/2}.
\end{equation}
\end{proposition}

To apply the above statement to our integrals, note that combining \eqref{IPL3} and \eqref{shift} gives
\begin{equation}\label{help1}
\begin{split}
\frac{n(n+\alpha+\beta+1)-m(m+\alpha+\beta+1)}{(1-x^2)w(x)}\int_x^1 P_n^{(\alpha,\beta)}(t)P_m^{(\alpha,\beta)}(t)\, w(t) dt=\\
\frac{n+\alpha+\beta+1}{2} P_{n-1}^{(\alpha+1,\beta+1)}(x)P_m^{(\alpha,\beta)}(x) -  \frac{m+\alpha+\beta+1}{2} P_{m-1}^{(\alpha+1,\beta+1)}(x)P_n^{(\alpha,\beta)}(x).
\end{split}
\end{equation}

Then, the latter formula and Proposition~\ref{mourad} lead to the following result.

\begin{theorem}
Let $\alpha$, $\beta$, $N$, and $\gamma$ be as in Proposition~\ref{mourad}. Suppose $n$ and $m$ are nonnegative integers, $m$ is fixed, and $n \neq m$. Then for a fixed $\theta\in(0, \pi)$ we have
\begin{equation}\label{TheAsymp}
    \int_{\cos\theta}^1 P_n^{(\alpha,\beta)}(t)P_m^{(\alpha,\beta)}(t)\, w(t)dt = \frac{\ell(\theta)}{n^{3/2}}\sin(N\theta + \gamma) + O(n^{-5/2})
\end{equation}
as $n\to\infty$, where 
\[
\displaystyle \ell(\theta) = \frac{2^{\alpha + \beta+1}}{\sqrt{\pi}}\left[\sin(\theta /2)\right]^{\alpha+1/2}\left[\cos(\theta/2)\right]^{\beta +1/2}P_m^{(\alpha,\beta)}(\cos\theta).
\] 
\end{theorem}

\begin{proof}
For convenience, let us set 
\[
I_n = \displaystyle \int_{\cos\theta}^1 P_n^{(\alpha,\beta)}(t)P_m^{(\alpha,\beta)}(t)\, w(t)dt.
\] 
Next, applying Proposition~\ref{mourad} to $P_{n-1}^{(\alpha+1,\beta+1)}$ and $P_n^{(\alpha,\beta)}$, we arrive at
\begin{equation} \label{mouradapplication1}
    P_{n-1}^{(\alpha+1,\beta+1)}(x) = \frac{[\sin(\theta/2)]^{-\alpha - 3/2}[\cos(\theta/2)]^{-\beta -3/2}}{\sqrt{(n-1)\pi}}\cos(N\theta + \gamma -\pi/2)+O(n^{-3/2})
\end{equation}
and 
\begin{equation} \label{mouradapplication2}
    P_n^{(\alpha,\beta)}(x) = \frac{[\sin(\theta/2)]^{-\alpha -1/2}[\cos(\theta/2)]^{-\beta - 1/2}}{\sqrt{n \pi}}\cos(N\theta + \gamma)+ O(n^{-3/2}),
\end{equation}
as $n\to\infty$ and where we set $x=\cos\theta$. Observe that $\cos(N\theta + \gamma - \pi/2) = \sin(N\theta + \gamma)$ and so \eqref{mouradapplication1} takes the form
\begin{equation} \label{mouradapplication12}
    P_{n-1}^{(\alpha+1,\beta+1)}(x) = \frac{[\sin(\theta/2)]^{-\alpha - 3/2}[\cos(\theta/2)]^{-\beta-3/2}}{\sqrt{(n-1)\pi}}\sin(N\theta + \gamma)+O(n^{-3/2}).
\end{equation}
Now, combing \eqref{mouradapplication2} and \eqref{mouradapplication12} with \eqref{help1}, we get
\begin{equation}\label{help2}
\begin{split}
\frac{(n(n+\alpha+\beta+1)-m(m+\alpha+\beta+1))I_n}{(1-x^2)w(x)}[\sin(\theta/2)]^{\alpha +3/2}[\cos(\theta/2)]^{\beta+3/2}=\\
\frac{n+\alpha+\beta+1}{2\sqrt{(n-1)\pi}}P_m^{(\alpha,\beta)}(x)\sin(N\theta + \gamma)+O(n^{-1/2}).
\end{split}
\end{equation}
Then, since $x = \cos \theta$, the weight function can be rewritten in the following way
\[
w(x) = (1-\cos\theta)^{\alpha}(1+\cos\theta)^{\beta} = 2^{\alpha+\beta}[\sin(\theta/2)]^{2\alpha}[\cos(\theta/2)]^{2\beta}. 
\]
Also, we have $1-x^2 = \sin^2\theta = 4[\sin(\theta/2)]^2[\cos(\theta/2)]^2$. As a result,  \eqref{help2} reduces to
\begin{equation*}%\label{help3}
\begin{split}
\frac{(n(n+\alpha+\beta+1)-m(m+\alpha+\beta+1))I_n}{2^{\alpha+\beta+2}[\sin(\theta/2)]^{\alpha +1/2}[\cos(\theta/2)]^{\beta+1/2}}=
\frac{\sqrt{n}}{2\sqrt{\pi}}P_m^{(\alpha,\beta)}(x)\sin(N\theta + \gamma)\\+O(n^{-1/2}),
\end{split}
\end{equation*}
%where
%\[
%A_n=\frac{n+\alpha+\beta+1}{2\sqrt{(n-1)\pi}}P_m^{(\alpha,\beta)}(x)
%\]
%and
%\[
%B_n=\sin(\theta/2)\cos(\theta/2)\frac{m+\alpha+\beta+1}{2\sqrt{n\pi}} P_{m-1}^{(\alpha+1,\beta+1)}(x).
%\]
%Define the number $\phi_n$ in the following manner
%\begin{equation}\label{phiDef}
%\sin\phi_n=-\frac{A_n}{ \sqrt{A_n^2 + B_n^2}}, \quad 
%\cos\phi_n=-\frac{B_n}{ \sqrt{A_n^2 + B_n^2}}
%\end{equation}
%and it allows us to write
%\begin{equation*}
%    A_n \sin(N\theta + \gamma) - B_n\cos(N\theta + \gamma) = \sqrt{A_n^2 + B_n^2}\cos(N\theta + \gamma + \varphi_n).
%\end{equation*}
%After that, one can see that 
%\[
% \sqrt{A_n^2 + B_n^2}=\frac{|P_m^{(\alpha,\beta)}(x)|}{2\sqrt{\pi}}\sqrt{n}+O(n^{-1/2}),
%\]
which yields \eqref{TheAsymp}.
\end{proof}

\noindent {\bf Acknowledgments.} M.D. was supported by the NSF DMS grant 2008844 and by the University of Connecticut Research Excellence Program.

\end{document}